\documentclass[a4paper, 11pt]{amsart}
\usepackage{a4wide}
\usepackage[english]{babel} 
\usepackage{amsmath, amssymb} 
\usepackage[all,cmtip]{xy} 
\usepackage{tikz}
\usepackage{tikz-cd}
\usepackage{xcolor}
\usepackage[normalem]{ulem}
\usepackage[mathscr]{euscript}
\usepackage{pifont}
\usepackage{hyperref}
\usepackage[capitalize,noabbrev]{cleveref}
\usepackage{dsfont}
\usepackage{stmaryrd}

\DeclareFontFamily{T1}{cbgreek}{}
\DeclareFontShape{T1}{cbgreek}{m}{n}{<-6>  grmn0500 <6-7> grmn0600 <7-8> grmn0700 <8-9> grmn0800 <9-10> grmn0900 <10-12> grmn1000 <12-17> grmn1200 <17-> grmn1728}{}
\DeclareSymbolFont{quadratics}{T1}{cbgreek}{m}{n}
\DeclareMathSymbol{\qoppa}{\mathord}{quadratics}{19}
\DeclareMathSymbol{\Qoppa}{\mathord}{quadratics}{21}

\newcommand{\B}{\mathrm{B}}

\newcommand{\Sp}{\mathrm{Sp}}

\newcommand{\K}{\mathrm{K}}
\newcommand{\s}{\mathrm{s}}

\newcommand{\gl}{\mathrm{Gl}}
\newcommand{\q}{\mathrm{q}}

\renewcommand{\L}{\mathrm{L}}

\newcommand{\MSTop}{\mathrm{MSTop}}

\newcommand{\MSG}{\mathrm{MSG}}
\newcommand{\BSG}{\mathrm{BSG}}

\newcommand{\ko}{\mathrm{ko}}
\newcommand{\KO}{\mathrm{KO}}

\newcommand{\KU}{\mathrm{KU}}

\newcommand{\id}{\mathrm{id}}

\newcommand{\n}{\mathrm{n}}

\newcommand{\Pic}{\mathrm{Pic}}

\newcommand{\Bsl}{\mathrm{BSl}}
\newcommand{\Bgl}{\mathrm{BGl}}
\newcommand{\Top}{\mathrm{Top}}
\newcommand{\PL}{\mathrm{PL}}
\newcommand{\alg}{\mathrm{alg}}
\renewcommand{\sl}{\mathrm{Sl}}
\newcommand{\bE}{\mathbb{E}}

\newcommand{\Spc}{\mathrm{Spc}}
\newcommand{\An}{\Spc}

\newcommand{\QF}{\Qoppa}

\newcommand{\Z}{\mathbb{Z}}

\renewcommand{\SS}{\mathbb{S}}

\newcommand{\map}{\mathrm{map}}
\newcommand{\Map}{\mathrm{Map}}

\newcommand{\sB}{\mathscr{B}}

\newcommand{\sC}{\mathscr{C}}
\newcommand{\D}{\mathscr{D}}

\newcommand{\ob}{\mathrm{ro}}
\newcommand{\pob}{\overline{\ob}}

\newcommand{\tso}{\mathrm{tso}}
\newcommand{\ptso}{\overline{\tso}}

\newcommand{\lto}{\longrightarrow}
\newcommand{\eps}{\epsilon}
\newcommand{\STop}{\mathrm{STop}}

\newcommand{\adj}{[\tfrac{1}{2}]}

\newcommand{\Fun}{\mathrm{Fun}}

\newcommand{\Th}{\mathrm{Th}}
\newcommand{\vr}{\mathrm{vr}}
\newcommand{\vs}{\mathrm{vs}}
\newcommand{\vq}{\mathrm{vq}}
\newcommand{\vn}{\mathrm{vn}}
\newcommand{\st}{\mathrm{st}}
\newcommand{\Baut}{\mathrm{Baut}}

\newcommand{\SP}{\mathrm{SP}}
\newcommand{\SN}{\mathrm{SN}}

\newcommand{\cS}{\mathscr{S}^\alg}
\newcommand{\pcS}{{\overline{\mathscr{S}}}^\alg}

\newcommand{\BTop}{\mathrm{BTop}}
\renewcommand{\S}{\mathbb{S}}

\renewcommand{\r}{\mathrm{r}}

\renewcommand{\ell}{l}
\swapnumbers

\DeclareMathOperator*{\colim}{colim}

\DeclareMathOperator*{\cofib}{cofib}
\DeclareMathOperator*{\fib}{fib}

\theoremstyle{plain}
\newcounter{zaehler}

\newtheorem*{introcor*}{Corollary}

\newtheorem{Thm}{Theorem}
\numberwithin{Thm}{section}
\newtheorem*{ThmA}{Main Theorem}
\newtheorem{Cor}[Thm]{Corollary}
\newtheorem{Lemma}[Thm]{Lemma}
\newtheorem{Prop}[Thm]{Proposition}

\theoremstyle{definition}

\newtheorem{Rmk}[Thm]{Remark}

\newtheorem{Preliminaries}[Thm]{Preliminaries}
\newtheorem*{Construction*}{Construction}

\newtheorem*{Acknowledgements}{Acknowledgements}

\theoremstyle{remark}

\newtheorem*{claim*}{Claim}

\title{Homology manifolds and euclidean bundles}

\author[F.~Hebestreit]{Fabian Hebestreit}
\address{Fakult\"at f\"ur Mathematik, Universit\"at Bielefeld, Germany} 
\email{hebestreit@math.uni-bielefeld.de}

\author[M.~Land]{Markus Land}
\address{Institut für Mathematik, JGU Mainz, Germany} 
\email{mland@uni-mainz.de}

\author[M.~Weiss]{Michael Weiss}
\address{Mathematisches Institut, Universit\"at M\"unster, Germany} 
\email{m.weiss@uni-muenster.de}

\author[C.~Winges]{Christoph Winges}
\address{Fakult\"at f\"ur Mathematik, Universit\"at Regensburg, Germany} 
\email{christoph.winges@ur.de}

\date{\today}

\begin{document}

\begin{abstract}
We construct a Poincar\'e complex whose periodic total surgery obstruction vanishes but whose Spivak normal fibration does not admit a reduction to a stable euclidean bundle. This contradicts the conjunction of two claims in the literature: Namely, on the one hand that a Poincar\'e complex with vanishing periodic total surgery obstruction is homotopy equivalent to a homology manifold, which appears in work of Bryant,Ferry, Mio and Weinberger, and on the other that the Spivak normal fibration of a homology manifold always admits a reduction to a stable euclidean bundle, which appears in work of Ferry and Pedersen. 
\end{abstract}

\maketitle
\tableofcontents

\section{Introduction}
In his celebrated work on algebraic surgery theory, Ranicki associates to an oriented Poincar\'e duality complex $X$ an invariant, the total surgery obstruction $\tso(X)$ of $X$, which for $\dim(X)\geq 5$ vanishes if and only if $X$ is realised by a closed, topological manifold \cite{RanickiTSO}. It encodes in a single invariant whether the Spivak normal fibration of $X$ admits a reduction to a stable euclidean bundle, and if so, whether any such reduction gives rise to a degree one normal map with trivial surgery obstruction. Through Siebenmann's famous periodicity mistake and the consequent hunt for the missing manifolds, the question arose whether there exists a similar total surgery obstruction for homology manifolds. This was finally established by Bryant, Ferry, Mio and Weinberger in \cite{BFMW} in terms of (what we shall refer to as) the periodic total surgery obstruction $\ptso(X)$ of $X$, a slightly coarser variant of Ranicki's $\tso(X)$. Based on earlier work of Ferry and Pedersen \cite{FP} that the Spivak normal fibration of a homology manifold always admits a reduction to a stable euclidean bundle, these authors proved that $\ptso(X) = 0$ if and only if $X$ is homotopy equivalent to a homology manifold, and established a surgery exact sequence to calculate the number of such realisations by a homology manifold. \\

The main goal of this note is, however, to show that the results of \cite{BFMW} and \cite{FP} are incompatible with one another by constructing Poincar\'e complexes $X$ with $\ptso(X) = 0$ whose Spivak normal fibration do not admit reductions to euclidean bundles. \\

To explain the construction, let us set up a bit of notation: We will write $\Bgl_1(\SS)$ for the classifying space of the group of units $\gl_1(\SS)$ of the sphere spectrum $\SS$. It can also be described as $\colim_{k \in \mathbb N} \mathrm{BAut}_{*}(S^{k})$ where $\mathrm{Aut}_*(S^k)$ denotes the topological monoid of pointed self-homotopy equivalences of $S^k$ and is hence the classifying space for stable spherical fibrations of virtual rank 0, classically denoted $\mathrm{BG}$. Its homotopy groups satisfy $\pi_n(\Bgl_1(\SS)) \cong \pi_{n-1}(\SS)$ for $n>1$.
The space $\Bgl_1(\SS)$ is the unit component of the Picard space $\mathrm{Pic}(\SS)$ of the sphere spectrum, i.e.\ the space of $\otimes$-invertible spectra, which classifies stable spherical fibrations of arbitrary virtual rank. The unit component of $\gl_1(\SS)$ will be denoted by $\sl_1(\SS)$, so that $\Bsl_1(\SS)$ classifies oriented stable spherical fibrations of virtual rank $0$.

\begin{Construction*}\label{Construction}
Let then $\bar{x} \in \pi_{\ell}(\Bgl_{1}(\SS))$ be an element of positive degree $\ell$ and order $n$. Let $M(\Z/n\Z,\ell)$ be the corresponding mod $n$ Moore space in dimension $\ell$. By assumption, there is then a map $x\colon M(\Z/n\Z,\ell) \to \Bgl_{1}(\SS)$ whose composite with the tautological map $S^{\ell} \to M(\Z/n\Z,\ell)$ is $\bar{x}$. 
By surgeries below the middle dimension one can, for any $m > 2\ell$, construct a stably framed closed smooth $m$-manifold $M$ and a map $f\colon M \to M(\Z/n\Z,\ell)$ inducing an isomorphism on $\pi_\ell$. For instance, one can perform framed surgery on an appropriately embedded $\ell$-sphere in $S^\ell \times S^{m-\ell}$ (mapping to the Moore space via projection to the first factor).
Let $k\geq \ell+1$ be large enough for $-xf\colon M \to \Bgl_{1}(\SS)$ to lift along $\mathrm{BAut}_{*}(S^{k}) \to \Bgl_{1}(\SS)$ and let $\pi \colon E \to M$ be the associated (pointed) spherical fibration over $M$. By choice of $k$, the resulting map $E \to M(\Z/n\Z,\ell)$ is then again an isomorphism on $\pi_\ell$. By \cite{Gottlieb} or \cite[Corollary F]{Klein}, $E$ is a Poincar\'e complex of dimension $d=m+k$, and by \cite{WH} its Spivak normal fibration $\mathrm D_E$ is classified by the composite 
\[E \xrightarrow{\pi} M \xrightarrow{xf} \{-d\} \times \Bgl_{1}(\SS) \subseteq \mathrm{Pic}(\SS);\]
here we view $xf$ as landing in the $-d$ component of $\Pic(\SS)$, as the Spivak fibration of $E$ has virtual rank $-d$.
\end{Construction*}

\begin{ThmA}\hypertarget{ThmB}
For $\bar{x} \in \pi_{l}(\Bgl_{1}(\SS))$ a non-trivial odd torsion element in the image of the $J$-homomorphism $\mathrm{BO} \rightarrow \Bgl_1(\SS)$, the Poincar\'e complexes $E$ constructed above satisfy $\ptso(E) = 0$ for any of the auxilliary choices, but their Spivak fibrations never admit reductions to euclidean bundles.
\end{ThmA}

Examples of elements $\bar{x}$ as required are for example given by any generator $\alpha_p$ of the $p$-torsion subgroup of $\pi_{2p-2}(\Bgl_{1}(\SS)) \cong \pi_{2p-3}(\SS)$; recall that this group is cyclic of order $p$ and that $2p-2$ is the smallest degree in which non-trivial $p$-torsion exists.
The simplest case is that of $\alpha_3$ and we give a separate argument for this specific case, as the argument simplifies considerably while still demonstrating the contradiction between \cite{BFMW} and \cite{FP}. 

In the final section, we will also explain the following addendum to the main result: For any odd torsion element $\bar{x} \in \pi_{l}(\Bgl_{1}(\SS))$ even the fibrewise periodic total surgery obstructions of the spherical fibrations $\pi \colon E \rightarrow M$ constructed above vanishes. This does not have direct implications in the direction of our main result, but should be related to the realisability of the fibration $E \rightarrow M$ by a block bundle of homology manifolds, which can be used as a test case for future developments.

\begin{Acknowledgements}
The authors heartily thank Shmuel Weinberger for helpful discussions. Moreover, we thank the referees for encouraging us to provide more details and the many valuable suggestions for improving the exposition of the paper.

FH, MW and CW were supported by the German Research Foundation (DFG) through the collaborative research centres ``Integral structures in Geometry and Representation theory'' (grant no.\ TRR 358--491392403) at the University of Bielefeld, ``Geometry: Deformations and Rigidity" (grant no.\ SFB 1442--427320536) at the University of M\"unster and ``Higher invariants" (grant no. \ SFB 1085--224262486) at the University of Regensburg, respectively. MW is furthermore a member of the cluster ``Mathematics M\"unster: Dynamics-Geometry-Structure" (grant no.\ EXC 2044--390685587) funded by the DFG at the University of M\"unster.
\end{Acknowledgements}

\section{Proof of the main result}

We start by recalling basic definitions and facts about the visible symmetric signatures and (periodic) total surgery obstructions of oriented Poincar\'e complexes, using the notational conventions from \cite{CDHI,CDHII}.
A reader familiar with these notions may safely skip the following list, in which $X$ usually refers to a general space, whereas $P$ refers to an oriented Poincar\'e complex. 
\begin{Preliminaries}
\begin{enumerate}
\item For a space $X$, we write $\L^\vr(X) = \L( (\D(\Z)_{/X})^\mathrm{f},\QF^\vr_\eps)$ for its \emph{visible quadratic} (the case $\r=\q$) and \emph{visible symmetric} (the case $\r=\s$) L-theory of Weiss and Ranicki; see \cite[Section 9]{Ranicki} where the visible symmetric L-groups are denoted $VL^*(\Z,X)$ in case $X$ is represented by a finite simplicial complex and \cite[Variant 4.4.15]{CDHI} for a general treatment. Here $(\D(\Z)_{/X})^\mathrm{f}$ denotes the full subcategory of the compact objects $(\D(\Z)_{/X})^\omega$ of $\D(\Z)_{/X} = \Fun(X,\D(\Z))$ on finite objects, i.e.\ those whose $\K$-class lies in the image of $\K_0(\Z) \to \K_0(\Z\pi_1(X)) \cong \K_0((\D(\Z)_{/X})^\omega)$. The subscript $\epsilon$ denotes the trivial rank $0$ spherical fibration over $X$ to align with notation from \cite{CDHI,CDHII}. The visible L-spectra are covariantly functorial in $X$; for $f\colon X \to Y$ we shall write $f_!\colon \L^\vr(X) \to \L^\vr(Y)$ for the induced map.

Essentially by definition, we have $\QF^{\mathrm{vq}}_\eps = \QF^\q_\eps$, so that visible quadratic L-theory is simply quadratic L-theory. Moreover, the algebraic $\pi$-$\pi$-theorem says that for $X$ pointed connected, the map $X \to B\pi_1(X)$ induces an isomorphism in visible quadratic L-theory, see \cite[Section 10]{Ranicki} and \cite[Corollary 1.2.33]{CDHIII}. The cofibre of the map $\L^\vq(X) \to \L^\vs(X)$ is denoted $\L^\vn(X)$ and called the \emph{visible normal} L-theory of $X$. In more classical terminology, the choice of working with finite objects corresponds to choosing decoration ``h'' for the L-spectra.
\item We have $\L^\vq(\ast)= \L^\q(\Z)$, $\L^\vs(\ast) = \L^\s(\Z)$, and $\L^\vn(\ast) = \L^\n(\Z)$, the usual quadratic, symmetric, and normal L-theory spectrum of the integers, the latter two of which are ring spectra. In fact, they are $\mathbb{E}_\infty$-ring spectra by the results of \cite{CDHIV}, but this refinement will not matter for the purposes of this paper. From $\L^\n_0(\Z) \cong \Z/8\Z$ \cite[pg.\ 13]{Ranicki} which follows from \cite[Prop.\ 7.2]{Ranicki-I} we see that $\L^\n(\Z)\adj = 0$ so that the map $\L^\q(\Z)\adj \to \L^\s(\Z)\adj$ is an equivalence. Note also that $\L^\q(\Z)$ does not refer to the $\q$-th symmetric L-group as it would in Ranicki's notation. 
\item When $P$ is an oriented Poincar\'e complex of dimension $d$, there is a canonical point $\sigma^\vs(P)$ in $\Omega^d \L^\vs(P)$, called the \emph{visible symmetric signature} of $P$, see \cite[Example 9.13]{Ranicki} and \cite[Corollary 4.4.20]{CDHI}. Its image $\sigma^\vn(P)$ in $\Omega^d \L^\vn(P)$ is called the \emph{visible normal signature} of $P$. These data fit into the following commutative square whose left vertical arrow is the canonical forgetful map between oriented Poincar\'e bordism and oriented normal bordism \cite{RanickiTSO}:
\[ \begin{tikzcd}
	\Omega^{\SP}(X) \ar[r] \ar[d] & \L^\vs(X) \ar[d] \\
	\Omega^{\SN}(X) \ar[r] & \L^\vn(X)
\end{tikzcd}\]
The upper horizontal map takes the (Poincar\'e bordism class of a) map $f\colon P \to X$ to $f_!(\sigma^\vs(P))$. In particular
$\sigma^\vs(P)$ is the image of $[\id_P]$ in $\Omega^\SP_n(P)$, see \cite[Section 19]{Ranicki}. Oriented normal bordism satisfies excision, giving an equivalence $\Omega^\SN(X) \simeq X \otimes \MSG$, and the lower horizontal arrow for $X=\ast$ is called \emph{Ranicki's normal orientation}, while the top horizontal arrow is often referred to as the \emph{Sullivan--Ranicki orientation}. Here, $X \otimes R$ denotes for a spectrum $R$ the $X$-indexed colimit of the constant diagram with value $R$. It is a spectrum whose homotopy groups are the $R$-homology groups of $X$.
\item\label{four} For an oriented closed topological manifold $M$, we may view $[\id_M]$ as an element of $\Omega^{\STop}_n(M)$, the oriented bordism of topological manifolds. By topological transversality \cite[Essay III, Section 1]{KS}, the Pontryagin--Thom map induces an equivalence $\Omega^\STop(X) \xrightarrow{\simeq} X \otimes \MSTop$ for any space $X$, showing that $\sigma^\vs(M)$ canonically lifts along the assembly map $M \otimes \L^\s(\Z) \to \L^\vs(M)$.
\item To better incorporate the orientation behaviour of the above elements, one considers $\L^\n_{1/2}(\Z) = \Z \times_{\Z/8\Z} \tau_{\geq0} \L^\n(\Z)$, the \emph{1/2-connective normal L-theory}. We then obtain canonical maps $\MSG \to \L^\n_{1/2}(\Z)$ as well as $\tau_{\geq 0} \L^\s(\Z) \to \L^\n_{1/2}(\Z)$ lifting the normal orientation $\MSG \to \L^\n(\Z)$ and the canonical projection $\L^\s(\Z) \to \L^\n(\Z)$. 
\item The assembly map in visible normal L-theory is an equivalence: $X \otimes \L^\n(\Z) \xrightarrow{\simeq} \L^\vn(X)$, see \cite[Section 15]{Ranicki}. We then write $\L^\vn_{1/2}(X) = X \otimes \L^\n_{1/2}(\Z)$ and denote by $\L^\vs_{1/2}(X)$ the pullback of $\L^\vs(X) \to \L^\vn(X) \leftarrow \L^\vn_{1/2}(X)$. Consider then the solid part of the diagram
\[\begin{tikzcd}
	\Omega^\SP(X) \ar[r, dashed] \ar[rr, bend left] \ar[d] & \L^\vs_{1/2}(X) \ar[d] \ar[r] & \L^\vs(X) \ar[d] \\
	\Omega^\SN(X) \ar[r] & \L^\vn_{1/2}(X) \ar[r] & \L^\vn(X)
\end{tikzcd}\]
In the lower horizontal factorisation the first map is induced from the lift $\MSG \to \L^\n_{1/2}(\Z)$ of the normal orientation from the previous point. The curved arrow is the Ranicki-Sullivan orientation from point (iii) above. Then we obtain a dashed arrow and think of the left hand square as a refined version of the square from point (iii).
Hence, the visible symmetric signature $\sigma^\vs(P)$ canonically lifts to a point $\sigma^\vs_{1/2}(P)$ of $\Omega^d\L^\vs_{1/2}(P)$. Its image $\sigma^\vn_{1/2}(P)$ in $\Omega^d\L^\vn_{1/2}(P)$ is called the \emph{normal fundamental class of $P$}.
\item For a space $X$, we write $\pcS(X) = \cofib[ X \otimes \L^\s(\Z) \to \L^\vs(X)]$ for the cofibre of the assembly map and $\cS(X) = \cofib[ X \otimes \tau_{\geq 0}\L^\s(\Z) \to \L^\vs_{1/2}(X)]$. There is a canonical map $\cS(X) \to \pcS(X)$ induced by the maps $\tau_{\geq0}\L^\s(\Z) \to \L^\s(\Z)$ and $\L^\vs_{1/2}(X) \to \L^\vs(X)$. 
\item The image of $\sigma^\vs_{1/2}(P)$ under the canonical map $\Omega^d \L^\vs_{1/2}(P) \to \Omega^d\cS(P)$ is denoted $\tso(P)$ and called the \emph{total surgery obstruction} of $P$. By point (iv) above, $\tso(P)$ vanishes in case $P$ is homotopy equivalent to a closed topological manifold, and the main result of \cite{RanickiTSO} says that for $\dim(P) \geq 5$, the converse holds as well.

The image $\ptso(P)$ of $\tso(P)$ under the canonical map $\Omega^d\cS(P) \to \Omega^d\pcS(P)$, the \emph{periodic total surgery obstruction}, is equally the image of $\sigma^\vs(P)$ under the canonical map $\Omega^d \L^\vs(P) \to \Omega^d\pcS(P)$. \cite[Main theorem]{BFMW} says that for $\dim(P) \geq 6$, $\ptso(P)$ is a complete obstruction for $P$ being homotopy equivalent to a closed homology manifold; but recall that the point of this note is to show that in this generality, the argument given in loc.\ cit.\ is not applicable, see also \cite{erratum}.
\item Consider next the diagrams:
\[ \begin{tikzcd}
	X \otimes \tau_{\geq1}\L^\q(\Z) \ar[r] \ar[d] & \L^\vq(X) \ar[d] & & X \otimes \L^\q(\Z) \ar[r] \ar[d] & \L^\vq(X) \ar[d] \\
	X \otimes \tau_{\geq0}\L^\s(\Z) \ar[r] &\L^\vs_{1/2}(X) & & X \otimes \L^\s(\Z) \ar[r] & \L^\vs(X)
\end{tikzcd}\]
The right hand square is a pullback since visible normal L-theory is excisive, and that the left hand square is a pullback then follows directly from the definition of $\L^\vs_{1/2}(X)$. As a consequence we obtain further commutative diagrams
\[\begin{tikzcd}
	\L^\vs_{1/2}(X) \ar[r] \ar[d] & \cS(X) \ar[d,"\partial"] && \L^\vs(X) \ar[r] \ar[d] & \pcS(X) \ar[d,"\partial"] \\
	\L^\vn_{1/2}(X) \ar[r,"-b"] & \Sigma[ X\otimes \tau_{\geq1}\L^\q(\Z)] && \L^\vn(X) \ar[r,"-b"] & \Sigma[X \otimes \L^\q(\Z)]
\end{tikzcd}\]
where $b$ denotes the boundary maps of the vertical and $\partial$ that of the horizontal maps. For a PD complex $P$, we then write $\ob(P) = \partial \tso(P)$ and $\pob(P) = \partial \ptso(P)$ and consequently find
\[ \ob(P) = -b(\sigma^\vn_{1/2}(P)) \quad \text{ and } \quad \pob(P) = -b(\sigma^\vn(P)).\]
By \cite{RanickiTSO}, the \emph{reduction obstruction} $\ob(P)$ is a complete obstruction for the Spivak normal fibration of $P$ to admit a reduction to a stable euclidean bundle. Recall now that by Atiyah-duality, see e.g.\ \cite[Lemma A.6]{Land} for this precise version, the Spanier--Whitehead dual of $P$ is canonically equivalent to $\Th(D_P)$, the Thom spectrum of its Spivak normal fibration $D_P$. Therefore, we obtain commutative diagrams as follows:
\[ \begin{tikzcd}[column sep=tiny]
	\L^\vn_{1/2}(P) \ar[r,"\simeq"] \ar[d,"b"] & \map(\Th(D_P),\L^\n_{1/2}(\Z)) \ar[d,"b"] && \L^\vn(P) \ar[r,"\simeq"] \ar[d,"b"] & \map(\Th(D_P),\L^\n(\Z)) \ar[d,"b"] \\
	\Sigma[P \otimes \tau_{\geq1}\L^\q(\Z)] \ar[r,"\simeq"] & \map(\Th(D_P),\Sigma\tau_{\geq1}\L^\q(\Z)) &&  \Sigma[P \otimes \L^\q(\Z)] \ar[r,"\simeq"] & \map(\Th(D_P),\Sigma\L^\q(\Z))
\end{tikzcd}\]
whose horizontal arrows are equivalences.
By general orientation theory, the normal fundamental class in $\Omega^d \L^\vn_{(1/2)}(P)$ corresponds under these horizontal equivalences to a Thom class in (1/2-connective) normal L-theory for the Spivak fibration $D_P$ of $P$. In particular, we see that $-\ob(P)$ and $-\pob(P)$ are equivalently described as the boundary operator applied to the normal L-theory Thom classes of $D_P$.
\item In his work on the total surgery obstruction \cite{RanickiTSO} Ranicki shows that the space of paths from $\tso(P)$ to $0$ inside $\Omega^{\infty+d} \cS(P)$ is canonically equivalent to the topological manifold block structure space $\widetilde{\mathscr{S}}(P)$. In particular, the topological block structure space $\widetilde{\mathscr{S}}(M)$ of a topological manifold $M$ identifies canonically with $\Omega^{\infty+d+1} \cS(M)$. In a similar fashion, the main results of \cite{BFMW} ought to give an identification of the homology manifold block structure space $\widetilde{\mathscr{S}}^{H}(P)$ with the space of paths from $\ptso(P)$ to $0$ inside $\Omega^{\infty+d} \pcS(P)$; for $P$ a topological manifold, this could indeed follow from careful inspection of the arguments in \cite{BFMW}, 
but for general $P$, our results below again imply that this last claim suffers from the defect indicated in point (viii).
\end{enumerate}
\end{Preliminaries}

We are now ready to prove our main result. As indicated, we first consider the special case in which $\bar{x}$ is a generator of the 3-torsion subgroup of $\pi_4(\Bgl_1(\SS)) \cong \Z/24\Z$, and then treat the general case.

\begin{proof}[Proof of \hyperlink{ThmB}{Main Theorem}]
As explained in the construction of $E$ in the introduction, the composite 
\[ E \xrightarrow{\pi} M \xrightarrow{f} M(\Z/3\Z,4) \xrightarrow{x} \{-d\} \times \Bgl_{1}(\SS) \]
classifies the Spivak normal fibration $D_E$ of $E$. It cannot factor through the map $\BTop \to \Bgl_{1}(\SS)$ since this is impossible on $\pi_{4}$: The composite induces an injection $\Z/3\Z \rightarrow \Z/24\Z$, but $\pi_{4}(\BTop) \cong \Z \oplus \Z/2$ \cite[Lemma 9]{Milgram}. Therefore, $D_E$ does not admit a reduction to a euclidean bundle.

It remains to show $\ptso(E)=0$. To that end, note first that $E$ is simply connected. As a consequence of the $\pi$-$\pi$-theorem, we thus find that the map
\[ \Omega^d \pcS(E) \to \Omega^{d-1}[E \otimes \L^\q(\Z)] \]
is split injective on homotopy groups.

 In particular, to verify that $\ptso(E) = 0$, it suffices to show that $\pob(E) = 0$. Since $D_E$ is classified by the composite $xf\pi$ as above, we obtain a map $\Th(D_E) \to \Th(x)$, where $\Th(-)$ is the Thom spectrum functor. 
Applying Spanier--Whitehead duality and Atiyah duality, we obtain an induced map $D(\Th(x)) \to D(\Th(D_E)) \simeq E$ which induces the following commutative diagram whose vertical maps issue from the boundary map $\L^\n(\Z) \to \Sigma \L^\q(\Z)$.
\[\begin{tikzcd}
	\pi_{d}(D(\Th(x)) \otimes \L^\n(\Z)) \ar[r] \ar[d] & \pi_{d}(D(\Th(D_E)) \otimes \L^\n(\Z)) \ar[d] \ar[r,"\cong"] & \pi_d(E \otimes \L^\n(\Z)) \ar[d] \\
	\pi_{d-1}(D(\Th(x)) \otimes \L^\q(\Z)) \ar[r] & \pi_{d-1}(D(\Th(D_E)) \otimes \L^\q(\Z)) \ar[r,"\cong"] & \pi_{d-1}(E \otimes \L^\q(\Z))
\end{tikzcd}\]
We have argued in point (ix) of the above list that $-\pob(E)$ is the image of the normal L-theory fundamental class $\sigma^\vn(E)$ of $E$ under the right vertical map, or equivalently the image of the normal L-theory Thom class of $D_E$ under the middle vertical map followed by the lower right horizontal isomorphism. Furthermore, as this normal Thom class in natural for oriented bundles it follows that it is the image of the normal L-theory Thom class of $x$ under the left most upper horizontal map. As a result, to show that $\pob(E)=0$, and therefore that $\ptso(E)=0$, it suffices to show that 
\[ \pi_{d-1}(D(\Th(x)) \otimes \L^\q(\Z)) = 0.\]
To that end, one first computes $D(\Th(x))$ as follows. In general, first note that for a pointed space $X$,
a pointed map $f\colon \Sigma X \to \Bgl_1(\SS)$ is by adjunction equivalent to a pointed map $X \to \mathrm{Gl}_1(\SS) \subseteq \Omega^\infty \SS$, and hence by another adjunction equivalent to a map $\tilde{f}\colon \Sigma^\infty X \to \SS$. Then, the Thom spectrum $\Th(f)$ identifies with the cofibre of the map $\tilde{f}\colon \Sigma^\infty X \to \SS$, for example the argument in \cite[Section A.13]{Land} applies to general suspensions in place of spheres. 
For the case at hand, we consider $x$ as a map $\Sigma M(\Z/3\Z,3) \to \{-d\} \times \Bgl_1(\SS)$ and then compose with the map $+d \colon \{-d\} \times \Bgl_1(\SS) \to \{0\} \times \Bgl_1(\SS)$ to obtain a map $y\colon \Sigma M(\Z/3\Z,3) \to \Bgl_1(\SS)$. Then we have $\Th(x) = \Omega^d \Th(y)$ and $\Th(y) = \cofib(\tilde{y})$ where $\tilde{y} \colon \SS^3/3 \to \SS$ by construction classifies a generator of the 3-torsion subgroup of $\pi_3(\SS)$; note here that $[\SS^3/3,\SS] \to [\SS^3,\SS] = \pi_3(\SS)$ is injective (since $\pi_4(\SS) = 0$) and has image the 3-torsion subgroup. As a result, there is an equivalence $\Th(x) \simeq \Omega^d\cofib(\SS^{3}/3 \xrightarrow{\tilde{y}} \SS)$ and hence 
\[ D(\Th(x)) \simeq \Sigma^d\fib(\SS \xrightarrow{D\tilde{y}} \SS^{-4}/3)\]
where $D\tilde{y}$ still classifies a generator of the 3-torsion of $\pi_3(\SS)$.
Consequently, we obtain an exact sequence
\[\pi_0(\L^\q(\Z)) \stackrel{D\tilde{y}}{\lto} \pi_4(\L^\q(\Z)/3) \lto \pi_{d-1}(D(\Th(x)) \otimes \L^\q(\Z)) \to 0 \]
with first group isomorphic to $\Z$ and middle one isomorphic to $\Z/3\Z$. We need to show that the map labelled $D\tilde{y}$ in this sequence is surjective. To do so, we may localise at $3$ and use the equivalence $\L^\q(\Z)_{(3)} \simeq \KO_{(3)}$, see \cite[Corollary 5.4]{LN} and the subsequent remark. It then suffices to show that the composite
\[ \SS \to \KO \xrightarrow{D\tilde{y}} \KO \otimes \SS^{-4}/3 \]
where the first map is the unit of the ring spectrum $\KO$, is surjective on $\pi_0$. But by construction, this map is $\Omega^4$ applied to a map
\[ \SS^4 \to \SS/3 \to \KO/3 \]
where the first map induces the projection map $\Z \to \Z/3\Z$ on $\pi_4$.
Then the result follows from the fact that the map $\S/3 \to \KO/3$ above, which is induced by the unit of $\KO$, induces a bijection on $\pi_4$ \cite[Thm.\ 1.7]{Adams}.
This completes the proof for the case of $\bar{x} \in \pi_4(\Bgl_1(\SS)) \cong \Z/24\Z$ a generator of the 3-torsion. 

Let us now explain the case of an arbitrary non-trivial element $\bar{x} \in \pi_{l}(\Bgl_1(\SS))$ of odd order $n$ which is in the image of the $J$-homomorphism. Note that this implies $l=4m$ for some $m$.
One first observes that \cite[Theorem 4.7 \& Remark 4.9]{Brumfiel}, together with the equivalence $\Top/\PL \simeq K(\Z/2,3)$ \cite[Essay V, Theorem 5.5]{KS}, says that the odd torsion of $\pi_{4m}(\BTop)$ maps isomorphically onto the kernel of Adams' $e$-invariant via the canonical map $\BTop \to \Bgl_1(\S)$. Since $e(\bar{x})$ is non-zero by \cite[Theorem 1.5 \& 1.6]{Adams} we deduce that the Spivak normal fibration of the Poincar\'e complex $E$ constructed from $\bar{x}$ does not admit a reduction to a euclidean bundle. 

To see that $\ptso(E) = 0$, with the same notation and arguments as above, it again suffices to show that the composite
\[ \S^{4m} \xrightarrow{D\tilde{y}} \S/n \to \KO/n \]
is surjective on $\pi_{4m}$, where again the second map is induced by the unit of $\KO$. To do this, one can reduce to the case where $n = p^j$ for some odd prime $p$ by the Chinese remainder theorem. Then we recall the fibre sequence
\begin{equation}\label{eq:K1-local-sphere}\tag{$\ast$}
\begin{tikzcd} L_{\KU/p}\S \ar[r,"u"] & \KO_p \ar[r,"\psi^t-1"] & \KO_p \end{tikzcd}
\end{equation}
where $t\in \Z_p^\times$ is a topological generator and $\psi^t$ the corresponding Adams operation. The map $u$ appearing in this sequence is then a ring map so that the mod $n$ unit map $\S/n \to \KO/n$ of $\KO$ factors as the composite $\S/n \to L_{\KU/p}\S/n \to \KO/n$. Hence, it suffices to show that in the composite
\[ \S^{4m} \to \S/n \to L_{\KU/p}\S/n \to \KO/n \]
the composite of the first two and the last map are each surjective on $\pi_{4m}$. To see the first claim, we consider the following diagram, whose vertical maps are the connecting maps in the mod $n$ reduction sequences.
\[ \begin{tikzcd}
	\S^{4m} \ar[r,"D\tilde{y}"] \ar[dr,"\overline x"'] & \S/n \ar[r] \ar[d] & L_{\KU/p}\S/n \ar[d] \\
		& \Sigma \S \ar[r] & \Sigma L_{\KU/p}\S
\end{tikzcd}\]
Now, the image of the $J$-homomorphism in $\pi_{4m-1}(\S)$ maps isomorphically onto $\pi_{4m-1}(L_{\KU/p}\S)$ as follows from \cite[Theorem 4.2.1]{Kochman}. In particular, the composite $\S^{4m} \to \Sigma L_{\KU/p}\S$ classifies an element of order $n$ in the cyclic group $\pi_{4m}(\Sigma L_{\KU/p}\S)$. Moreover, it follows from the long exact sequence in homotopy groups associated to the fibre sequence \eqref{eq:K1-local-sphere}, the fact that $t \in \Z_p^\times$ has infinite order, and the fact that $\psi^t$ is a ring map that $\pi_{4m}(L_{\KU/p}\S)=0$. Thus the rightmost vertical map in the above diagram is injective on $\pi_{4m}$ with image the $n$-torsion subgroup. As an element of order $n$ in a cyclic $n$-torsion group, the top horizontal composite therefore represents a generator of $\pi_{4m}(L_{\KU/p}\S/n)$ and is thus surjective on $\pi_{4m}$.

It now remains to show that the map $L_{\KU/p}\S/n \to \KO/n$ induces a surjection on $\pi_{4m}$. Tensoring the fibre sequence \eqref{eq:K1-local-sphere} with $\S/n$ and investigating the long exact sequence on homotopy groups, it remains to show that $\psi^t-1$ induces the zero map on $\pi_{4m}(\KO/n)$. 
Since the map $\pi_{4m}(\KO_p) \to \pi_{4m}(\KO/n)$ is isomorphic to the standard surjection $\Z_p \to \Z/n$, it suffices to argue that the effect of $\psi^t-1$ on $\pi_{4m}(\KO_p)$, which is multiplication by a number $s$, is divisible by $n$. But $s$ is the order of $\pi_{4m-1}(L_{\KU/p}\S)$, and $n$ is the order of an element in this group, so indeed $n$ divides $s$. This finishes the proof.
\end{proof}

\begin{Rmk}\label{Rmk:ob(E)}
Let us briefly note what the above argument says about the (non-periodic) reduction obstruction $\ob(E)$, which we know to be non-zero since the Spivak normal fibration of $E$ does not admit a reduction to a stable euclidean bundle. Running the same argument as in the above proof, we are lead to consider the commutative diagram
\[\begin{tikzcd}
	\pi_d(D(\Th(x))\otimes \L^\n_{1/2}(\Z)) \ar[r] \ar[d] & \pi_d(E\otimes \L^\n_{1/2}(\Z)) \ar[d] \\
	\pi_{d-1}(D(\Th(x))\otimes \tau_{\geq1}\L^\q(\Z)) \ar[r] & \pi_{d-1}(E\otimes \tau_{\geq 1}\L^\q(\Z))
\end{tikzcd}\]
and note that the element $\ob(E)$ is the image of the normal L-theory Thom class of $x$ along either of the two displayed composites. We may then compute the lower left corner analoguously as before, and obtain that it sits in an exact sequence
\[ \pi_0(\tau_{\geq1}\L^\q(\Z)) \to \pi_{4m}(\tau_{\geq1}\L^\q(\Z)/n) \to \pi_{d-1}(D(\Th(x))\otimes \tau_{\geq1}\L^\q(\Z)) \to 0\]
where now the first term is $0$ rather than $\Z$ as in the case of the periodic reduction obstruction considered above, and the second term is isomorphic to $\Z/n\Z$ as $m\geq 1$. This shows that $\ob(E)$ is the image of an $n$-torsion element along the lower horizontal composite in the above diagram.
\end{Rmk}

\section{Consequences and addenda}

We collect a few statements that lie in the vicinity of our arguments above.

\subsubsection*{Relation to the literature}
\begin{enumerate}
\item Our \hyperlink{ThmB}{Main Theorem} invalidates the use of the statement, which appears in the proof of the main result of \cite{BFMW}, that a Poincar\'e complex with vanishing periodic total surgery obstruction has a reducible Spivak fibration, see \cite{erratum}.

If, \cite[Main theorem]{BFMW} is correct as stated, then the Poincar\'e complexes $E$ constructed before the proof of the \hyperlink{ThmB}{Main Theorem} are homotopy equivalent to closed homology manifolds whose Spivak normal fibrations do not admit reductions to stable euclidean bundles provided $\bar{x} \in \pi_\ell(\Bgl_1(\SS))$ has odd order and is in the image of the $J$-homomorphism. This would invalidate \cite[Theorem 16.6]{FP}, which we regard as the more likely resolution. At this point we do, however, not have any further insight to offer. In particular, we do not know whether our complexes $E$ are realised by homology manifolds, but see the final two items in this remark.
\item The discussion of Quinn's resolution obstruction for homology manifolds in \cite[Section~3]{BFMW} also uses that the Spivak normal fibration admits a reduction to a euclidean bundle, and hence suffers the same defect as \cite[Main Theorem]{BFMW}.
\item In \cite{erratum}, a sufficient condition for the existence of a euclidean reduction for Poincar\'e complexes with trivial periodic total surgery obstruction is mentioned. Since the method of proof is very similar to what we have explained above, we record a general statement and proof in \cref{ThmA} below.
\item We observe in \cref{Prop:2-patch} below that the Poincar\'e complexes $E$ we construct admit a $\pi$-$\pi$-2-patch structure. An argument for the main result of \cite{BFMW} for Poincaré complexes admitting such structures would therefore suffice to show our complexes are in fact realised by closed homology manifolds; Weinberger has indicated to us that this special case may be significantly simpler than the general one.
\item We will show in \cref{Cor:fibrewise-tso-0} below that even the fibrewise periodic total surgery obstructions of the fibrations $\pi \colon E \to M$ constructed before the \hyperlink{ThmB}{Main Theorem} vanish. Assuming that the family version of \cite[Main theorem]{BFMW}, see point (xi) above, is correct even only for fibrations whose typical fibre is a topological manifold it follows that $\pi$ is homotopy equivalent to block bundles with typical fibre a homology manifold. 
Again, this would certainly imply that $E$ is homotopy equivalent to a homology manifold. 
\end{enumerate}
In the remainder of this note, we discuss points (iii), (iv), and (v) above in more detail.

\subsubsection*{Positive results about euclidean reductions.}

The relation between vanishing of the periodic and non-periodic reduction obstructions can be described fairly directly:

\begin{Prop}\label[Prop]{ThmA}
Let $X$ be an oriented $d$-dimensional Poincar\'e complex such that $\pob(X)=0$. Then the following statements are equivalent:
\begin{enumerate}
\item The Spivak normal fibration of $X$ admits a reduction to a stable euclidean bundle,
\item there exists a degree one map $M \to X$ where $M$ is an oriented closed topological manifold, and
\item the fundamental class $[X] \in H_d(X;\Z\adj)$ lifts to a fundamental class in $\ko_d(X)\adj$.
\end{enumerate}
\end{Prop}
\begin{proof}
An argument of Sullivan's shows that (i) implies that there exists a degree one \emph{normal} map $M \to X$ for some closed topological manifold $M$, see \cite[Prop.\ 10.2]{Wall-book} or \cite[Theorem 7.19]{Lueck} for the smooth case. Using topological transversality \cite[Essay III, Section 1]{KS} the same argument applies to the topological case. In particular, (i) implies (ii). Given (ii), we find that $[M \to X]$ determines an element of $\MSTop_d(X)$ which lifts the fundamental class $[X] \in H_d(X;\Z)$. In particular, after inverting $2$, one may use the Sullivan--Ranicki orientation $\MSTop\adj\to \tau_{\geq0}\L^\s(\Z)\adj \simeq \ko\adj$ to lift the fundamental class in $H_d(X;\Z\adj)$ to $\ko_d(X)\adj$; see again \cite[Corollary 5.4 \& following Remark]{LN} for the equivalence $\L^\s(\Z)\adj \simeq \KO\adj$ of $\mathbb{E}_\infty$-ring spectra.

Finally, consider the exact sequence
\[ \pi_d(X \otimes \L^\q(\Z)) \xrightarrow{(\ast)} \pi_d(X \otimes \tau_{\leq0}\L^\q(\Z)) \to \pi_{d-1}(X\otimes \tau_{\geq1}\L^\q(\Z)) \to \pi_{d-1}(X\otimes \L^\q(\Z)) \]
and recall that the second to last group contains the element $\ob(X)$ which is sent to $\pob(X)=0$ in the final group above. Note that the map labelled $(\ast)$ is 2-locally surjective, as follows from the 2-local splitting of $\L^\q(\Z)$ into Eilenberg-Mac Lane spectra. Away from $2$, we may use the equivalence $\L^\q(\Z)\adj \simeq \KO\adj$ and consider the following commutative diagram.
\[ \begin{tikzcd}
	\pi_d(X\otimes \ko)\adj \ar[r] \ar[d,"\cong"] & \pi_d(X \otimes \tau_{\leq0}\ko)\adj \ar[r,"\cong"] \ar[d,"\cong"] & H_d(X;\Z\adj) \\
	\pi_d(X \otimes \KO)\adj \ar[r] & \pi_d(X\otimes \tau_{\leq0}\KO)\adj
\end{tikzcd}\]
Since $X$ has trivial homology in degrees $>d$, it follows from the Atiyah--Hirzebruch spectral sequence that for any spectrum $E$, the canonical map $\pi_d(X\otimes \tau_{\geq0}E) \to \pi_d(X \otimes E)$ is an isomorphism. In particular, in the above diagram, the vertical maps are isomorphisms. Hence, under assumption (iii) we deduce the map $(\ast)\adj$, which is isomorphic to the lower horizontal map in the above diagram, is also surjective. This shows that $\ob(X) =0$ and hence (i).
\end{proof}
We note that $\pob(X) = 0$ is implied by $\ptso(X) = 0$, in which case the following also appears in \cite{erratum}.
\begin{Cor}
Let $X$ be an oriented Poincar\'e complex with $\pob(X) =0$. If $\dim(X) \leq 6$, then the Spivak normal fibration admits a reduction to a stable euclidean bundle.
\end{Cor}
\begin{proof}
It follows from work of Thom \cite{Thom} that any homology class of degree $\leq 6$ is represented by a smooth, in particular by a topological, closed manifold, so condition (ii) of \cref{ThmA} is satisfied.
\end{proof}

\begin{Rmk}
In the above argument, we have used that for an oriented Poincar\'e complex $X$, $\pob(X) =0$ implies that $\ob(X)$ vanishes 2-locally. One may wonder whether the same is true for the total surgery obstructions: If $\ptso(X) = 0$, does it follow that $\tso(X)$ vanishes 2-locally? We do not know the answer to this question currently. Note, however, that the 2-local splitting of $\L^\q(\Z)$ does not induce a $2$-local splitting of $\pcS(X)$ into $\cS(X)$ and a second summand: Indeed, suppose that $X$ is a $d$-dimensional aspherical Poincar\'e complex and $\pi_1(X)$ is a Farrell--Jones group. Then $\pcS(X) = 0$ but $\cS(X)_{(2)} \simeq [X \otimes  \tau_{\leq0}\L^\q(\Z)]_{(2)} \neq 0$. Nevertheless, at least conjecturally, one has $\tso(X)=0$ whenever $X$ is aspherical, see \cite[Conjecture 9.178]{Lueck-FJ}. 
\end{Rmk}

\subsubsection*{Patch structures}
Let us recall that a $\pi$-$\pi$-2-patch structure on a Poincar\'e complex $X$ consists of two topological manifolds $M$ and $M'$ whose boundary inclusions $\partial M \subseteq M$ and $\partial M' \subseteq M'$ are equivalences on fundamental groupoids (that is, on $\pi_0$ and $\pi_1$ for all basepoints), and homotopy equivalences $\varphi \colon \partial M \simeq \partial M'$ as well as $\psi \colon X \simeq M \cup_\varphi M'$. A more general notion of patch structures on Poincar\'e complexes was introduced and studied by Jones \cite{Jones}.
\begin{Prop}\label[Prop]{Prop:2-patch}
The Poincar\'e complexes $E$ constructed above admit $\pi$-$\pi$-2-patch structures provided $\ell\geq 3$.
\end{Prop}
\begin{proof}
We note that $M(\Z/n\Z,\ell) \cong \Sigma M(\Z/n\Z,\ell-1)$ and consider a transverse inverse image $N \subseteq M$ of $M(\Z/n\Z,\ell-1)$ along the map $f \colon M \to M(\Z/n\Z,\ell)$ appearing in \cref{Construction}. By construction $N \subseteq M$ is a codimension one submanifold with trivial normal bundle. Since $M(\Z/n\Z,\ell-1)$ is simply connected, by ambient surgeries on $N$, one can achieve that $N$ is simply connected as well, see \cite{Quinn2} for general results along these lines. Consequently, $M\setminus N$ decomposes into a disjoint union of two pieces which are the interiors of compact manifolds $M^+$ and $M^{-}$ with common boundary $N$. Moreover, $M^\pm$ is again simply connected, so that the inclusion $N \subseteq M^{\pm}$ is an equivalence on fundamental groupoids.
In other words, $M$ is obtained from two manifolds $M^{+}$ and $M^{-}$ by glueing along a diffeomorphism $\partial M^{+} \cong \partial M^{-}$. 
Moreover, the fibration classified by $M \to \mathrm{Baut}_{*}(S^{k})$ is trivial when restricted to $M^{\pm}$ as it then factors through the contractible upper/lower cones of the suspension $\Sigma M(\Z/n\Z,l-1)$.
It follows that $E$ is obtained from the two manifolds $M^{+} \times S^{k}$ and $M^{-}\times S^{k}$ by glueing along a homotopy equivalence $\partial M^{+} \times S^{k} \simeq \partial M^{-}\times S^{k}$, which is the promised $\pi$-$\pi$-2-patch structure on $E$.
\end{proof}

\subsubsection*{The fibrewise periodic total surgery obstruction}
Finally, we aim to show that even the fibrewise periodic total surgery obstructions of the fibrations we have constructed vanish. We will use a number of observations about them that we explain first. 

\begin{Preliminaries}\label[Construction]{Construction:fibrewise-orientation}
Let $P$ be an oriented Poincaré complex. Whenever $R$ is a ring spectrum under $\MSG$ the space $P$ is equipped with a canonical orientation class in $R$-homology, that is a point $\sigma^R(P)$ in $\Omega^d[ P\otimes R]$, capping with which induces Poincar\'e duality in $R$-(co)homology. We will mostly use $R=\L^\n(\Z)$ via the normal Sullivan--Ranicki orientation, in which case we denoted this orientation class by $\sigma^\vn(P)$ earlier. 

Assume now that $\pi\colon E \to B$ is an oriented $P$-fibration. On account of the $\MSG$ orientation the fundamental classes of its fibres assemble into a section $\sigma^R(\pi)$ of the parametrised spectrum $\Omega^d[E\otimes_B R] = \Omega^d \pi_!\pi^*r^*(R)$ over $B$, where $r\colon B \to \ast$ denotes the unique map. Fibrewise Poincar\'e duality in $R$-(co)homology identifies this parametrised spectrum with $\pi_*\pi^*r^*(R)$, whose fibre over some $b \in B$ is the space of maps $\Map(E_b,R)$. Therefore, the space of sections of $\Omega^d[E\otimes_B R]$ identifies canonically with $(r\pi)_*(r\pi)^*(R) \simeq \Map(E,R)$ and under this equivalence, the fibrewise orientation class $\sigma^R(\pi)$ is the point classified by the unit of the adjunction $R \to (r\pi)_*(r\pi^*)$. 

If $\pi$ is equipped with a section $s$, we will refer to the image of $\sigma^R(\pi)$ under the map on section spaces induced by the projection $\Omega^d[E \otimes_B R] \rightarrow \Omega^d[(E,s(B)) \otimes_B R]$ as the reduced fibrewise orientation class. 
\end{Preliminaries}

Let us specialise this to the case where $P=S^d$ is a sphere. We denote by $\pi_\st$ the fibre of the canonical augmentation $[E \otimes_B \S] \to [B \otimes_B \S]$ induced by $\pi$, equivalently, of the counit map $\pi_!\pi^*(\S_B) \to \S_B$. In more classical language $\pi_\st$ is classified by the composite 
\[B \to \Baut^+(S^d) \to \Baut_*^+(S^{d+1}) \to \BSG \times \{d+1\} \to \BSG \times \{d\}, \]
where the latter map is induced by $-1 \in \Z = \pi_0(\Pic(\S))$, the middle map is induced by $\Sigma^\infty \colon \An_* \to \Sp$ and the first map is induced by $\Sigma \colon \An \to \An_*$. 

By \cite{WH}, the parametrised spectrum $\pi^*\pi_\st^{-1}$ is the relative dualising spectrum of $\pi$.
We will require the explicit identification below, so we give a short proof. To this end, note that using the projection formula for $\pi_!$, we have the map
\[ c_\pi \colon \SS_B = \pi_\st \otimes_B \pi_\st^{-1} \to \pi_!\pi^*(\SS_B) \otimes \pi_\st^{-1} = \pi_!\pi^*(\pi_\st^{-1}). \]
where the middle one is induced by the tautological map $\pi_\st \to \pi_!\pi^*(\S_B)$ that is part of the defining fibre sequence of $\pi_\st$. 
\begin{Lemma}\label[Lemma]{dualising-spectrum}
	The natural transformation
	\[ \pi_*(-) \xrightarrow{c_\pi} \pi_*(-) \otimes_B \pi_!\pi^*(\pi_\st^{-1}) = \pi_!(\pi^*\pi_*(-) \otimes_E \pi^*\pi_\st^{-1}) \to \pi_!(- \otimes_E \pi^*\pi_\st^{-1}) \]
	induced by the projection formula for $\pi_!$ and the counit $\pi^*\pi_* \to \id$ is an equivalence.
\end{Lemma}
\begin{proof}
	It suffices to check that this transformation is a fibrewise equivalence.
	So let $b \colon * \to B$ be arbitrary, let $E_b \simeq S^d$ denote the fibre of $\pi$ over $b$, and let $p \colon E_b \to *$ as well as $j \colon E_b \to E$ be the canonical maps.
	Since the projection formula is compatible with base change, applying $b^*$ yields the natural transformation
	\[ p_*j^*(-) \xrightarrow{b^*c_\pi} p_*(-) \otimes p_!p^*b^*\pi_\st^{-1} = p_!(p^*p_*(-) \otimes_{E_b} p^*b^*\pi_\st^{-1}) \to p_!(- \otimes_{E_b} p^*b^*\pi_\st^{-1}). \]
	The map $b^*c_\pi$ is the inclusion $\SS \to \SS \oplus \SS^{-d}$, so it exhibits $p^*b^*\pi_\st^{-1} \simeq \SS^{-d}_{E_b}$ as the dualising spectrum of $E_b$.
	This implies that the displayed map is an equivalence as required.
\end{proof}

Since $\pi_*$ also satisfies the projection formula, it follows from \cite[Theorem~2.35]{cnossen:twisted-ambidexterity} (applied to the case $\sB = \An_{/B}$, $\sC = \Sp$, $A = \pi$ and $B = \id_B$) that the pair $(\pi^*\pi_\st^{-1},c_\pi)$ is the unique pair $(V,c)$ with $V$ a parametrised spectrum over $E$ and $c\colon \S_B \to \pi_!(V)$ satisfying the analog of \cref{dualising-spectrum} (this equivalence is our definition of the relative dualising spectrum).

Suppose now that $\pi$ is equipped with a section $s$. Then $s$ induces a section $f\colon \S_B \to \pi_!\pi^*(\S_B)$ of the counit. In particular, we may identify the cofibre of $f$ with $\pi_\st$
\[ \begin{tikzcd}
	\pi_\st \ar[r,"\iota"'] & \pi_!\pi^*(\S_B) \ar[r] \ar[l,bend right,"g"'] & \S_B \ar[l,"f"', bend right]
\end{tikzcd}\]
such that $g$ is a retraction of $\iota$.

We now unravel the reduced fibrewise orientation class $\bar{\sigma}^R(\pi)$. It is given by the composite
\[ \S_B \xrightarrow{c_\pi \otimes_\S 1_R} \pi_!\pi^*(\pi_{\st}^{-1}\otimes_\S R) \simeq \Omega^d \pi_!\pi^*(R_B) 
\xrightarrow{g\otimes_\S R} \Omega^d \pi_\st \otimes R \simeq R_B \]
where the two equivalences are induced by the canonical $\MSG$-orientation of $\pi_\st$ (which induces equivalences $\Omega^d\S_B \otimes_\S R \to \pi_\st^{-1}\otimes_\S R$ and $\Sigma^d \S_B \otimes_\S R \to \pi_\st \otimes_\S R$).
Unravelling the definitions, this map is given by the unit of $R$: 
using again projection formulas as in the definition of $c_\pi$, this is equivalent to showing that the composite
\[ \pi_\st \xrightarrow{\iota} \pi_!\pi^*(\S_B) \to \pi_!\pi^*(R_B) \xrightarrow{g\otimes_\S R} \pi_\st\otimes R\]
is the unit of $R$, which follows from the definition of $g$.
\newline

From the functoriality of $\pcS(-)$, we obtain a parametrised spectrum $\pcS(\pi)$ over $B$ sending a point $b \in B$ to $\pcS(E_b)$. The periodic total surgery obstructions $\ptso(E_b)$ of the fibres $E_b$ then assemble into a section $\ptso(\pi)$ of $\Omega^d\pcS(\pi)$.

As indicated above, the family version of \cite[Main theorem]{BFMW} for closed topological manifold fibres asserts  that an oriented Poincar\'e fibration $\pi\colon E \rightarrow B$ with typical fibre a topological manifold $F$, $\pi$ is fibre homotopy equivalent to a block bundle with typical fibre a homology manifold (homotopy equivalent to $F$) if and only if the section $\ptso(\pi)$ just described vanishes. In particular, when $\ptso(\pi)$ vanishes and $B$ is a manifold, $E$ is then homotopy equivalent to a homology manifold. In the following, we will give sufficient conditions which imply that $\ptso(\pi)$ vanishes.

\begin{Lemma}\label[Lemma]{Lemma}
Let $\pi\colon E \to B$ be  an oriented Poincar\'e fibration which has simply connected fibres and is equipped with a section $s$. Then the following are equivalent.
\begin{enumerate}
\item Giving a trivialisation of the fibrewise periodic total surgery obstruction $\ptso(\pi)$, and
\item giving a lift of the fibrewise relative normal orientation $\bar{\sigma}^\vn(\pi)$ to a section of the parametrized spectrum $\Omega^{d}[(E,s(B)) \otimes_B \L^\s(\Z)]$.
\end{enumerate}
\end{Lemma}
\begin{proof}
We consider the following diagram of $B$-parametrised spectra: (we write everywhere the value of these objects over a point $b$, and note that all constructions are functorial in $b$).
\[\begin{tikzcd}
	 & \L^\vq(E_b,s(b)) \ar[d] & \\
	(E_b,s(b)) \otimes \L^\s(\Z) \ar[r] \ar[d] & \L^\vs(E_b,s(b)) \ar[r] \ar[d,"\simeq"] & \pcS(E_b) \\
	(E_b,s(b)) \otimes \L^\n(\Z) \ar[r,"\simeq"] & \L^\vn(E_b,s(b)) 
\end{tikzcd}\]
Here, the symbols in the middle vertical row mean the cofibres of the maps induced by the inclusion $s(b) \to E_b$.
The horizontal and vertical composites are each fibre sequences. This is because $\pcS(E_b)$ is the cofibre of the assembly map in (visible) symmetric L-theory, which is, by construction, an equivalence on the basepoint $s(b)$. We then note that the algebraic $\pi$-$\pi$-theorem implies that the quadratic L-theory term on the very top vanishes as we assume $E_b$ to be simply connected. Moreover, we recall that the assembly map in visible normal L-theory is an equivalence. We therefore find the two indicated equivalences. Passing to spaces of sections we obtain similar fibre sequences and equivalences. Now, by construction, the canonical section of $b \mapsto (E_b,s(b))\otimes \L^\n(\Z)$ corresponds under the two indicated equivalences to the canonical section of $b \mapsto \L^\vs(E_b,s(b))$ taking the fibrewise symmetric signature of the bundle $E \to B$. Therefore, a lift of the left vertical map on spaces of sections is equivalent to the datum of a nullhomotopy of the image of the symmetric signature section in $\Gamma(B,\pcS(E))$, which, by definition, is $\ptso(\pi)$.
\end{proof}

\begin{Thm}\label[Thm]{prop:zack}
Let $\pi\colon E \to B$ be an oriented pointed spherical fibration with typical fibre homotopy equivalent to $S^d$ whose associated element $\pi_\st$ in $[B,\Bsl_1(\SS)]$ is of odd order. Then there exists a trivialization of $\ptso(\pi)$.
\end{Thm}
\begin{proof}
Working with each connected component of $B$ separately, we may assume that $B$ is connected.
We will show that the data in (ii) of \cref{Lemma} can be given. We recall that $\bar{\sigma}^\vn(\pi)$ is a section of the parametrised $\L^\n(\Z)$-module spectrum $b \mapsto \Omega^d[(E_b,s(b))\otimes \L^\n(\Z)]$. In \cref{Construction:fibrewise-orientation}, we have argued that this parametrized $\L^\n(\Z)$-module spectrum is classified by a shift of $\pi_\st$:
\[ B \xrightarrow{\pi} \mathrm{BAut}_*^+(S^d) \to \{d\}\times \Bsl_1(\S) \xrightarrow{-d} \{0\}\times \Bsl_1(\S) \to \Bgl_1(\L^\n(\Z))\]
where the final map is induced by the unit of the ring spectrum $\L^\n(\Z)$. We also noted that this map is null-homotopic by the tautological $\MSG$-orientation of $\pi_\st$. Since it will become important now, we recall what this tautological orientation is. First, we note that $\L^\n(\Z)$ is 2-local so we may work 2-locally whenever suitable.
By orientation theory, the composite
\[ \Bsl_1(\SS_{(2)}) \to \Bgl_1(\SS_{(2)}) \to \Bgl_1(\MSG_{(2)}) \] 
is canonically trivialised as a map of $\mathbb{E}_\infty$-groups since $\MSG_{(2)}$ is the Thom spectrum of the $\bE_\infty$-map $\Bsl_1(\S_{(2)}) \to \Bgl_1(\S_{(2)})$ so that the claim follows from \cite[Lemma 3.15]{ACB}, using \cite[Def.\ 3.14]{ACB} and the pullback square appearing in the beginning of the proof of \cite[Prop.\ 3.16]{ACB} (for $R=\S_{(2)}$ and $A=\MSG_{(2)}$).
It follows that the null-homotopy of the above composite is induced by the canonical null-homotopy of the fibre sequence
\[ \Bsl_1(\SS_{(2)}) \to \Bgl_1(\SS_{(2)}) \to \B\Z_{(2)}^\times \]
and a map $\B\Z_{(2)}^\times \to \Bgl_1(\MSG_{(2)})$. Upon applying the loop functor, this latter map sends a 2-local unit $q$ to the self map of $\MSG_{(2)}$ given by multiplication with $q$. Consequently, the composite
\[ B \to \Bsl_1(\S) \to \Bgl_1(\MSG_{(2)}) \to \Bgl_1(\L^\n(\Z)), \]
which classifies the parametrised spectrum $\Omega^d((E,s(B)) \otimes_B \L^\n(\Z))$ as explained above, is canonically null-homotopic as claimed. Here, the second map is induced by the unit of the ring spectrum $\MSG_{(2)}$ and the final map by the normal Sullivan--Ranicki orientation $\MSG \to \L^\n(\Z)$. As explained in \cref{Construction:fibrewise-orientation}, using this null-homotopy, the reduced fibrewise normal fundamental class $\bar{\sigma}(\pi)$ becomes the tautological section which is constant at $1 \in \pi_0\L^\n(\Z)$.

Next, we note that the assumption that $\pi$ thought of as an element of $[B,\Bsl_1(\SS)]$ is of odd order is equivalent to the statement that the composite
\[ B \xrightarrow{\pi} \Bsl_1(\SS) \to \Bsl_1(\SS_{(2)}) \]
is null-homotopic. Choosing such a null-homotopy, we obtain two null-homotopies of the composite 
\[ B \to \Bsl_1(\SS) \to \B\Z_{(2)}^\times \] 
which together determine a map $q\colon B \to \Z_{(2)}^\times$. The composite
\[ B \to \Bsl_1(\SS) \to \Bgl_1(\MSG_{(2)}) \to \Bgl_1(\L^\n(\Z)) \]
is therefore also null-homotopic in two ways, whose difference is determined by the composite
\[ B \xrightarrow{q} \Z_{(2)}^\times \to \gl_1(\MSG_{(2)}) \to \gl_1(\L^\n(\Z))\]
where we may think of $q$ as an element of $\Z_{(2)}^\times$ as we have reduced to the case where $B$ is connected.
This implies that the composite equivalence of $B$-parametrised $\L^\n(\Z)$-module spectra
\[\begin{tikzcd} r^*\L^\n(\Z) \ar[r,"\simeq"',"(1)"] & \Omega^d[(E,s(B))\otimes_B \L^\n(\Z)] & r^*\L^\n(\Z)\ar[l,"\simeq","(2)"'] \end{tikzcd} \]
is simply given by multiplication by $q$. Here, equivalence (1) comes from orientation theory as discussed in the beginning of the proof while equivalence (2) comes from the assumption that $\pi_\st$ has odd order as an element of $[B,\Bsl_1(\S)]$. It follows that under equivalence (2), the normal orientation $\bar{\sigma}^\vn(\pi)$ corresponds to the map $B \to \Omega^\infty\L^\n(\Z)$ which is constant at $q \in \pi_0 \L^\n(\Z)$. Now, we consider the following diagram
\[\begin{tikzcd}
	\Omega^d[(E,s(B))\otimes_B \L^\s(\Z)_{(2)}] \ar[d] & r^*\L^\s(\Z)_{(2)} \ar[d] \ar[l,"\simeq"] \\
	\Omega^d[(E,s(B))\otimes_B \L^\n(\Z)] & r^*\L^\n(\Z) \ar[l,"\simeq","(2)"']
\end{tikzcd}\]
where the upper horizontal equivalence is again induced by the trivialisation of the composite $B \to \Bsl_1(\S) \to \Bsl_1(\S_{(2)})$, and is hence compatible with equivalence (2) discussed above. Under the top horizontal equivalence, the map $q \colon B \to \Omega^\infty\L^\s(\Z)_{(2)}$ which is constant at $q \in \pi_0\L^\s(\Z)$ therefore lifts the normal orientation $\bar{\sigma}^\vn(\pi)$ to a section $\bar{s}$ of the $B$-parametrised spectrum $\Omega^d[(E,s(B))\otimes_B \L^\s(\Z)_{(2)}]$. 
Now, since $B$ is compact and $\L^\s(\Z)_{(2)}$ is the colimit over multiplication by odd integers on $\L^\s(\Z)$, there exists an odd integer $\alpha$ and a section $s$ of the $B$-parametrised spectrum $\Omega^d[(E,s(B))\otimes_B \L^\s(\Z)]$ such that $\bar{s}$ is the image of $s$ under the map
\[ \Gamma(\Omega^d[(E,s(B))\otimes_B \L^\s(\Z)]) \xrightarrow{\frac{\mathrm{can}}{\alpha}} \Gamma(\Omega^d[(E,s(B))\otimes_B \L^\s(\Z)_{(2)}])\]
induced on spaces of sections by the canonical map to the 2-localisation divided by $\alpha$. Consider finally the section $\alpha \cdot s$ of $\Omega^d[(E,s(B))\otimes_B \L^\s(\Z)]$. We claim that it lifts the normal orientation $\bar{\sigma}^\vn(\pi)$, finishing the proof of the theorem. Indeed, its image in $\Gamma(\Omega^d[(E,s(B))\otimes_B \L^\n(\Z))$ is, by construction, given by $\alpha^2$ times the image of $\bar{s}$ in $\Gamma(\Omega^d[(E,s(B))\otimes_B \L^\n(\Z)])$, hence under the equivalence induced by (2) by the map $\alpha^2 \cdot q \colon B \to \Omega^\infty \L^\n(\Z)$. But now we may use that $\pi_0\L^\n(\Z) \cong \Z/8\Z$ and that $\alpha^2 \equiv 1$ modulo $8$ to obtain the desired result.
\end{proof}

\begin{Cor}\label{Cor:fibrewise-tso-0}
For the spherical fibrations $\pi\colon E \to M$ constructed from $\bar{x} \in \pi_\ell(\Bgl_1(\SS))$ before the \hyperlink{ThmB}{Main Theorem}, we have $\ptso(\pi) = 0$ provided the order $n$ of $\bar{x}$ is odd.
\end{Cor}
\begin{proof}
By \cref{prop:zack}, it suffices to show that the classifying map of $\pi$, viewed as an element in $[M,\Bsl_1(\SS)]$, is of odd order. But by construction, this classifying map is in the image of the map 
$[M(\Z/n\Z,\ell),\Bsl_1(\SS)] \to [M,\Bsl_1(\SS)]$ and the source of this map is a finite group of odd order.
\end{proof}

\bibliographystyle{amsalpha}
\bibliography{mybib}

\end{document}